\documentclass[11pt]{article}

\usepackage{graphicx}
\usepackage{amsthm,amsmath,amssymb,tikz,bm}
\usepackage{mathtools}
\usepackage{tikz-network}
\usepackage{xifthen}
\usepackage{comment}
\usepackage{thmtools}
\usepackage[margin=1.0in]{geometry}
\usepackage[shortlabels]{enumitem}
\usepackage{todonotes}
\usepackage[colorlinks=true,
linkcolor=blue,citecolor=blue,
urlcolor=blue]{hyperref}

\usetikzlibrary{calc,shapes,backgrounds,positioning}

\tikzset{
    comp/.style={
        draw,
        rounded corners=2pt,
        minimum width=9mm,
        minimum height=18mm,
        align=center
    },
    smallcomp/.style={
        draw,
        minimum width=10mm,
        minimum height=8mm,
        align=center
    },
    blob/.style={
        draw,
        ellipse,
        minimum width=5.8cm,
        minimum height=2.2cm
    },
    region/.style={
        draw,
        ellipse,
        minimum width=11mm,
        minimum height=22mm,
        align=center
    },
    vtx/.style={
        circle,
        fill,
        inner sep=1.3pt
    }
}

\newtheorem{theorem}{Theorem}[section]
\newtheorem{lemma}[theorem]{Lemma}
\newtheorem{corollary}[theorem]{Corollary}

\newtheorem{claim}{Claim}

\newtheorem{proposition}[theorem]{Proposition}
\newtheorem{question}{Question}



\makeatletter
               {\list{}{\leftmargin=0pt
                        \labelwidth\z@ \itemindent-\leftmargin
                        }}%
               {\endlist}
\makeatother

\title{Perfect Divisibility, Linear Divisibility and Chair-Free Graphs}

\author{
Zhiyu Wang \thanks{Department of Mathematics, Louisiana State University,
Baton Rouge, LA 70803 (\texttt{zhiyuw@lsu.edu}). This author was supported
in part by LA Board of Regents grant LEQSF(2024--27)-RD-A-16.}
\and
Weihao Xia \thanks{Department of Mathematics, Louisiana State University,
Baton Rouge, LA 70803 (\texttt{wxia3@lsu.edu}).}
}

\begin{document}

\maketitle

\begin{abstract}
A graph is \emph{perfectly divisible} if every induced subgraph with at
least one edge admits a partition into a perfect induced subgraph and an
induced subgraph with smaller clique number. Every perfectly divisible
graph $G$ satisfies $\chi(H)\leq\binom{\omega(H)+1}{2}$ for every induced
subgraph $H$ of $G$. We show that the converse fails: for every non-negative
integer $t$, the graph $P(17)\vee K_t$ satisfies this bound for every induced
subgraph but is not perfectly divisible, yielding an infinite family of counterexamples. Motivated by this distinction, we
introduce \emph{$(k,\ell)$-linear divisibility} and prove that every
$(k,\ell)$-linearly divisible graph $G$ satisfies
$\chi(G)\leq k\binom{\omega(G)+1}{2}$. As an application of this framework,
we give a direct structural decomposition showing that every chair-free graph
is $(2,2)$-linearly divisible, where a chair is obtained from $K_{1,3}$ by
subdividing one edge once. The chair-free result was obtained
independently before we became aware of a recent preprint of Liu, Sun, Wang,
Wu, and Zeng~\cite{LSWWZ2026}, who prove the stronger statement that every
chair-free graph is perfectly weight divisible and hence satisfies
$\chi(G)\leq\binom{\omega(G)+1}{2}$.
\end{abstract}

\section{Introduction}\label{sec:intro}

Relating the chromatic number $\chi(G)$ of a graph $G$ to its clique number
$\omega(G)$ is a central topic in structural graph theory. Since
$\chi(G)\geq\omega(G)$, a fundamental problem is to identify graph classes
for which $\chi(G)$ can be bounded above by a function of $\omega(G)$. A
class of graphs $\mathcal G$ is called \emph{hereditary} if every induced
subgraph of every graph in $\mathcal G$ also belongs to $\mathcal G$. One
important and well-studied family of hereditary classes is given by
\emph{$H$-free graphs}, that is, graphs with no induced subgraph isomorphic
to a fixed graph $H$. More generally, given a class of graphs $\mathcal H$,
we say that a graph $G$ is \emph{$\mathcal H$-free} if $G$ is $H$-free for
every $H\in\mathcal H$.
By constructions of Tutte~\cite{De47,De54} and Mycielski~\cite{Mycielski1955}
(see also~\cite{Scott-Seymour2020} for further constructions), there exist
triangle-free graphs with arbitrarily large chromatic number. Hence, for the
class of all graphs, there is no general upper bound on the chromatic number
in terms of the clique number. A graph class $\mathcal G$ is called
\emph{$\chi$-bounded} if there exists a function $f$, called a
\emph{$\chi$-binding function}, such that $\chi(G)\leq f(\omega(G))$ for
every $G\in\mathcal G$. A graph class is called \emph{polynomially
$\chi$-bounded} if it admits a polynomial $\chi$-binding function.
Gy\'arf\'as~\cite{Gyarfas1975} and Sumner~\cite{Sumner1981} conjectured,
in 1975 and 1981, respectively, that for every tree $T$, the class of
$T$-free graphs is $\chi$-bounded. This conjecture is known only for
relatively few classes of forbidden trees~\cite{CSS2019,Gyarfas1975,
GST1980,Gyarfas1987,Kierstead-Penrice1994,Kierstead-Zhu2004,SSS2022,
Scott1997,Scott-Seymour2020}. Polynomial
$\chi$-boundedness is particularly important because, if the class of
$H$-free graphs is polynomially $\chi$-bounded, then $H$ satisfies the
Erd\H{o}s--Hajnal conjecture~\cite{Erdos-Hajnal1977,Erdos-Hajnal1989},
which asserts that there exists $\epsilon_H>0$ such that every $H$-free graph
$G$ has a clique or stable set of size at least $|V(G)|^{\epsilon_H}$.

A classical example of a hereditary $\chi$-bounded class is the class of
\emph{perfect graphs}: a graph $G$ is perfect if every induced subgraph $H$
of $G$ satisfies $\chi(H)=\omega(H)$. Thus, perfect graphs admit the identity
function as a $\chi$-binding function. Chudnovsky, Robertson, Seymour, and
Thomas~\cite{CRST2006} characterized perfect graphs as the graphs with no odd
hole and no odd antihole. This is the Strong Perfect Graph Theorem.
A graph $G$ is called \emph{perfectly divisible} if, for every induced
subgraph $H$ of $G$ with at least one edge, the vertex set $V(H)$ can be
partitioned into sets $A$ and $B$ such that $H[A]$ is perfect and
$\omega(H[B])<\omega(H)$. 
Perfect divisibility was introduced by Ho\`ang~\cite{Hoang2018} in his study
of $(\mathrm{banner},\mathrm{odd\ hole})$-free graphs, and it generalizes
perfect graphs. Indeed, every perfect graph is perfectly divisible, by taking
$A=V(H)$ and $B=\emptyset$ for every induced subgraph $H$ with at least one
edge. Moreover, a
standard induction on $\omega(G)$ shows~\cite{CS2018} that every perfectly
divisible graph $G$ satisfies
$\chi(G)\leq\binom{\omega(G)+1}{2}$. Thus, perfect divisibility yields a
quadratic $\chi$-binding function. Existing proofs that certain graphs are not perfectly divisible have often
proceeded by showing that this chromatic bound fails; see, for
example,~\cite{CLZ2026, LSWWZ2026}. It is therefore natural to ask
whether the converse holds. This question was suggested to us by
V.~Sivaraman~\cite{Sivaraman2026PC}.

\begin{question}\label{question:perfect-divisible}
Is a graph $G$ perfectly divisible if and only if
$\chi(H)\leq\binom{\omega(H)+1}{2}$ for every induced subgraph $H$ of $G$?
\end{question}

We first answer Question~\ref{question:perfect-divisible} in the negative.
Let $q$ be a prime power with $q\equiv1\pmod 4$, and let $\mathbb F_q$ be
the finite field of order $q$. The \emph{Paley graph} $P(q)$ has vertex set
$\mathbb F_q$, where two distinct vertices $x,y\in\mathbb F_q$ are adjacent
if and only if $x-y$ is a nonzero square in $\mathbb F_q$. Equivalently,
\[
E(P(q))
=
\bigl\{\{x,y\}:x,y\in\mathbb F_q,\ x\neq y,\
 x-y\in(\mathbb F_q^\times)^2\bigr\},
\]
where $(\mathbb F_q^\times)^2=\{a^2:a\in\mathbb F_q^\times\}$.

\begin{restatable}{theorem}{paleySeventeen}\label{p17}
The Paley graph $P(17)$ satisfies
$\chi(H)\leq\binom{\omega(H)+1}{2}$ for every induced subgraph $H$ of
$P(17)$, but $P(17)$ is not perfectly divisible. Consequently, every graph
containing $P(17)$ as an induced subgraph is not perfectly divisible.
\end{restatable}

Let $\mathcal{G}_{\mathrm{quad}}$ denote the hereditary class of graphs $G$
such that $\chi(H)\leq \binom{\omega(H)+1}{2}$
for every induced subgraph $H$ of $G$. By Theorem~\ref{p17},
$P(17)\in\mathcal{G}_{\mathrm{quad}}$, but $P(17)$ is not perfectly
divisible. Therefore, every graph in $\mathcal{G}_{\mathrm{quad}}$ that
contains $P(17)$ as an induced subgraph is a counterexample to
Question~\ref{question:perfect-divisible}.
For two vertex-disjoint graphs $G$ and $F$, let $G\vee F$ denote their
\emph{join}, obtained from their disjoint union by adding all edges between
$V(G)$ and $V(F)$. In particular, for every non-negative integer $t$, the graph
$P(17)\vee K_t$ belongs to $\mathcal{G}_{\mathrm{quad}}$ (where $K_0$ is considered as empty graph). 
Since $P(17)\vee K_t$ contains $P(17)$ as an induced subgraph, it is not
perfectly divisible. Thus,
$\bigl\{P(17)\vee K_t:t\in\mathbb{Z}_{\geq 0}\bigr\}$ is an infinite family
of connected counterexamples to Question~\ref{question:perfect-divisible}.

\medskip
\noindent\textbf{Remark 1.}
Let $\mathcal F=\{P(17)\vee K_t:t\in\mathbb Z_{\geq 0}\}$.
Since $\alpha(P(17))=3$ and
$\alpha(G\vee F)=\max\{\alpha(G),\alpha(F)\}$, every graph in
$\mathcal F$ has independence number $3$ and, by
Theorem~\ref{p17}, is not perfectly divisible. Thus, $\mathcal F$ is another
explicit infinite family of counterexamples to a conjecture of
Ho\`ang~\cite{Hoang2026}, which asserts that every graph $G$ with
$\alpha(G)\leq3$ is perfectly divisible; see also~\cite{CLZ2026} for an earlier counterexample.
For positive integers $a,b,c$, let $S_{a,b,c}$ be the tree formed by three
paths of lengths $a,b,c$ that have one common end and are otherwise
vertex-disjoint. Since each of $K_{1,4}$, $S_{1,1,3}$, and $S_{2,2,2}$
has independence number $4$, every graph in $\mathcal F$ is also
$K_{1,4}$-free, $S_{1,1,3}$-free, and $S_{2,2,2}$-free. Thus, the same
family gives another explicit infinite family satisfying the
three conclusions of
Liu, Sun, Wang, Wu, and Zeng~\cite[Proposition~8.3]{LSWWZ2026}.
Moreover, the hereditary closure $\mathcal C$ of $\mathcal F$ admits the
linear $\chi$-binding function $f(\omega)=\omega+3$. Indeed, every graph in $\mathcal C$ has the form $H\vee K_s$,
where $H$ is an induced subgraph of $P(17)$, and
$\chi(H\vee K_s)-\omega(H\vee K_s)
=\chi(H)-\omega(H)\leq3$. 
\medskip

Theorem~\ref{p17} shows that perfect divisibility is strictly stronger
than the hereditary quadratic chromatic bound defining
$\mathcal{G}_{\mathrm{quad}}$. Motivated by this distinction, we introduce
\emph{$(k,\ell)$-linear divisions}. Unlike perfect divisibility, a
$(k,\ell)$-linear division does not require any part of the graph to induce a
perfect graph. Instead, it uses pairwise-complete induced subgraphs of
smaller clique number together with a bound on the chromatic number of the
remaining induced subgraph.

For a nonnegative integer $r$, let $[r]=\{1,\ldots,r\}$, where
$[0]=\emptyset$. Let $k$ and $\ell$ be positive integers, and let $G'$ be a
nonempty graph with $\omega(G')=\omega$. A partition
$V(G')=V(G_1)\sqcup V(G_2)$, where $G_1$ and $G_2$ are induced subgraphs
of $G'$, is called a \emph{$(k,\ell)$-linear division} of $G'$ if there exist
an integer $r\geq0$ and $\alpha\in\{0,\ldots,\min\{\ell,\omega\}\}$ such
that:
\begin{enumerate}[(1)]
\item
$V(G_1)=V(D_1)\sqcup\cdots\sqcup V(D_r)$, where
$D_1,\ldots,D_r$ are nonempty induced subgraphs of $G'$ that are pairwise
complete to one another, and $\omega(G_1)\leq\omega-\alpha$. Moreover, if
$r=1$, then $\alpha\geq1$; if $r\geq2$, then $\omega(D_i)\geq2$ for every
$i\in[r]$; while if $r=0$, then $G_1$ is empty and
$\alpha=\min\{\ell,\omega\}$.

\item
\[
\chi(G_2)\leq
\begin{cases}
 k\alpha\left(\omega-\dfrac{\alpha-1}{2}\right),
 & r\in\{0,1\},\\[6pt]
 k\left((\alpha+2r-2)
 \left(\omega-\dfrac{\alpha+2r}{2}\right)
 +\dfrac{3\alpha}{2}\right),
 & r\geq2.
\end{cases}
\]
\end{enumerate}
A graph $G$ is called \emph{$(k,\ell)$-linearly divisible} if every connected
induced subgraph of $G$ admits a $(k,\ell)$-linear division. A graph class
$\mathcal{G}$ is called \emph{uniformly $k$-linearly divisible} if there
exists a positive integer $\ell$ such that every graph in $\mathcal{G}$ is
$(k,\ell)$-linearly divisible. We use connected induced subgraphs so that
disjoint unions of $(k,\ell)$-linearly divisible graphs are again
$(k,\ell)$-linearly divisible. If $\ell'\geq \ell$, then every
$(k,\ell)$-linearly divisible graph is $(k,\ell')$-linearly divisible, since the
direct bound for $r=0$ is increasing in $\alpha$.

\medskip
\noindent\textbf{Remark 2.}
The parameter $\ell$ bounds $\alpha$ and therefore restricts the divisions
that may be used. In particular, if $r=0$ and $\omega\geq\ell$, then
condition~(2) becomes
$\chi(G')\leq k\ell\left(\omega-\frac{\ell-1}{2}\right)$, which is linear
in $\omega$ for fixed $\ell$. Thus, the hereditary quadratic chromatic
bound alone does not imply $(k,\ell)$-linear divisibility via $r=0$.
The conclusion of Theorem~\ref{thm:kld} does not depend on $\ell$ because,
for every admissible choice of $\alpha$ and $r$, the bounds in
condition~(2) are chosen so that the inductive bound on $G_1$ and the
direct bound on $G_2$ add to at most
$k\binom{\omega+1}{2}$. Consequently, smaller values of $\ell$ define a
more restrictive property, while every value of $\ell$ yields the same
quadratic chromatic bound.
\medskip

Every perfectly divisible graph is $(1,1)$-linearly divisible. Indeed, let
$G'$ be a connected induced subgraph of a perfectly divisible graph, and
write $\omega=\omega(G')$. If $\omega=1$, then $r=0$ and $\alpha=1$ give
the required division. Assume that $\omega\geq2$, and let
$V(G')=A\sqcup B$ be a perfect division. If $B=\emptyset$, take $r=0$ and
$\alpha=1$; then $G'=G'[A]$ is perfect and $\chi(G')=\omega$. Otherwise,
take $G_1=D_1=G'[B]$, $G_2=G'[A]$, $r=1$, and $\alpha=1$. Then
$\omega(G_1)\leq\omega-1$, while $G_2$ is perfect and
$\chi(G_2)\leq\omega$. Thus, perfect divisibility implies
$(1,1)$-linear divisibility. The converse does not hold by
Proposition~\ref{prop:p17-linear}.

We show by standard induction that every $(k,\ell)$-linearly divisible graph satisfies the following quadratic chromatic upper bound.
\begin{restatable}{theorem}{linearlydivisiblebound}\label{thm:kld}
Let $k$ and $\ell$ be positive integers, and let $G$ be a
$(k,\ell)$-linearly divisible graph with $\omega(G)=\omega$. Then
$\chi(G)\leq k\binom{\omega+1}{2}$.
\end{restatable}

Despite extensive effort, even the asymptotics of the optimal $\chi$-binding
functions for $T$-free graphs are known for very few trees $T$. It is known
that $K_{1,t}$-free graphs and $P_4$-free graphs are polynomially
$\chi$-bounded. Indeed, if $G$ is $K_{1,t}$-free, then
$\chi(G)\leq\Delta(G)+1\leq R(\omega(G),t)$, while every $P_4$-free graph is
perfect. A natural common superclass of these two classes is the class of
\emph{$t$-broom-free graphs}, where a $t$-broom is obtained from
$K_{1,t+1}$ by subdividing one edge once. A $2$-broom is also called a
\emph{chair}. Schiermeyer and Randerath~\cite{Schiermeyer-Randerath2019}
asked whether chair-free graphs admit a polynomial $\chi$-binding function.

\begin{figure}[htb]
\hbox to \hsize{
    \hfil
    \resizebox{2cm}{!}{\begin{tikzpicture}[scale=1, Wvertex/.style={circle, draw=black, fill=white, scale=2}, bvertex/.style={circle, draw=black, fill=black, scale=0.2},rvertex/.style={circle, draw=red, fill=red, scale=0.2}]

\node [bvertex, label={[font=\small] right:$v_2$}] (v2) at (1,-1) {};
\node [bvertex, label={[font=\small] right:$v_1$}] (v1) at (1,0) {};
\node [bvertex, label={[font=\small] left:$u_0$}] (u0) at (0,0)  {};
\node [bvertex, label={[font=\small] left:$u_1$}] (u1) at (0,-1) {};
\node [bvertex, label={[font=\small] left:$u_2$}] (u2) at (0,1) {};

\draw (v2) -- (v1);
\draw (v1) -- (u0);
\draw (u0) -- (u1);
\draw (u0) -- (u2);

\end{tikzpicture}	}%
    \hfil
    \resizebox{3.5cm}{!}{\begin{tikzpicture}[scale=1, Wvertex/.style={circle, draw=black, fill=white, scale=2}, bvertex/.style={circle, draw=black, fill=black, scale=0.2},rvertex/.style={circle, draw=red, fill=red, scale=0.2}]

\node [bvertex, label={[font=\small] above:$v_2$}] (v2) at (-2,0) {};
\node [bvertex, label={[font=\small] above:$v_1$}] (v1) at (-1,0) {};
\node [bvertex, label={[font=\small] above:$u_0$}] (u0) at (0,0)  {};
\node [bvertex, label={[font=\small] right:$u_1$}] (u1) at (1,1) {};
\node [bvertex, label={[font=\small] right:$u_2$}] (u2) at (1,0.5) {};
\node [font=\scriptsize] at (1,0) {$\rotatebox{90}{$\cdots$}$};

\node [bvertex, label={[font=\small] right:$u_{t-1}$}] (u3) at (1,-0.5) {};
\node [bvertex, label={[font=\small] right:$u_{t}$}] (u4) at (1,-1) {};

\draw (v2) -- (v1);
\draw (v1) -- (u0);
\draw (u0) -- (u1);
\draw (u0) -- (u2);
\draw (u0) -- (u3);
\draw (u0) -- (u4);
\end{tikzpicture}	}%
    \hfil
}
    \caption{The chair and the $t$-broom.}
    \label{fig:broom-graph}
\end{figure}
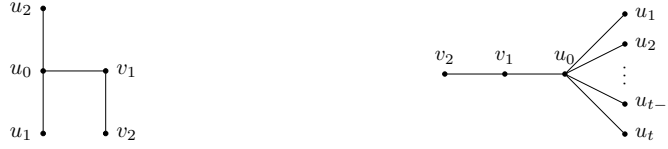

Recently, Liu, Schroeder, Wang, and Yu~\cite{LSWY2023} established a polynomial
$\chi$-binding function $f(\omega)=o(\omega^{t+1})$ for $t$-broom-free
graphs with clique number $\omega$. For $t=2$, they further gave the bound
$f(\omega)=7\omega^2$ for chair-free graphs. For additional coloring bounds on subclasses of chair-free graphs, see e.g.,~\cite{ChudnovskyCookSeymour2020, ChudnovskyHuangKarthickKaufmann2021}. Sivaraman~\cite{KKS2022} further
conjectured that every chair-free graph is perfectly divisible. 
Chudnovsky and Sivaraman~\cite{CS2018} introduced the weighted analogue of
perfect divisibility. For a positive integral weight function
$h:V(G)\to\mathbb Z_{>0}$ and an induced subgraph $H$ of $G$, let
$\omega_h(H)$ denote the maximum total $h$-weight of a clique in $H$. A
graph $G$ is \emph{perfectly weight divisible} if, for every such function
$h$ and every induced subgraph $H$ of $G$ with at least one edge, there is a
partition $V(H)=A\sqcup B$ such that $H[A]$ is perfect and
$\omega_h(H[B])<\omega_h(H)$. Taking $h\equiv1$ shows that perfect weight
divisibility implies perfect divisibility. 
Perfect divisibility
was established for several subclasses of chair-free graphs. Karthick,
Kaufmann, and Sivaraman~\cite{KKS2022} proved that every
$(\mathrm{chair},F)$-free graph is perfectly divisible for
$F\in\{P_6,\mathrm{co\mbox{-}dart},\mathrm{bull}\}$, and Wu and
Xu~\cite{WuXu2024} proved the same for $F=\mathrm{odd\ balloon}$.
Xu and Zhuang~\cite{XZ2026} showed that perfect weight divisibility of
chair-free graphs is equivalent to that of claw-free graphs and proved that
$(\mathrm{chair},F)$-free graphs are perfectly weight divisible for
$F\in\{P_7,P_6\cup K_1\}$. Chen and Zhang~\cite{ChenZhang2026} proved
perfect divisibility for
$F\in\{\mathrm{odd\ balloon}^{+},P_3\cup P_4\}$. Lan, Liu, Wu, and
Zhou~\cite{LanLiuWuZhou2026} showed that every
$(\mathrm{chair},\mathrm{odd\ parachute})$-free graph is either perfectly
divisible or has a trisimplicial vertex, and consequently satisfies the
quadratic chromatic bound implied by perfect divisibility.

Below, improving the $7\omega^2$ upper bound established by Liu, Schroeder, Wang and Yu~\cite{LSWY2023}, we show that the class of chair-free graphs is uniformly $2$-linearly divisible.
\begin{restatable}{theorem}{twolinearlydivisible}\label{main thm}
Every chair-free graph is $(2,2)$-linearly divisible. Consequently, if $G$
is a chair-free graph with clique number $\omega$, then
$\chi(G)\leq2\binom{\omega+1}{2}$.
\end{restatable}

\noindent\textbf{Note added.}
Theorem~\ref{main thm} and its proof were obtained independently before we became aware of the recent preprint of Liu, Sun, Wang, Wu, and Zeng~\cite{LSWWZ2026}. Building on Xu and Zhuang's reduction of perfect weight divisibility for chair-free graphs to the claw-free case, they apply the Chudnovsky--Seymour global structure theorem for claw-free graphs~\cite{CSClawIV,CSClawV} to prove the stronger statement that every claw-free graph is perfectly weight divisible, thereby confirming Sivaraman's conjecture. Their result implies that every chair-free graph is perfectly divisible, is thus $(1,1)$-linearly divisible, and satisfies
$\chi(G)\leq\binom{\omega(G)+1}{2}$,
and therefore supersedes the chromatic consequence of
Theorem~\ref{main thm}. We retain our theorem because its proof uses a
fundamentally different approach, analyzing chair-free graphs directly
through a refined template method and a structural decomposition, and
illustrates how such a decomposition can be converted into linear divisions.

\medskip
\noindent\textbf{Notation and terminology.}
All graphs in this paper are finite and simple. Given a graph $G$ and a
vertex $v\in V(G)$, let $N_G(v)$ denote the neighborhood of $v$, let
$N_G[v]=N_G(v)\cup\{v\}$, and let $d_G(v)=|N_G(v)|$. For a positive integer
$i$ and a set $S\subseteq V(G)$, let
$N_G^i(S)
=
\{u\in V(G)\setminus S:\min\{d_G(u,v):v\in S\}=i\}$,
where $d_G(u,v)$ denotes the distance between $u$ and $v$ in $G$. Thus,
$N_G^1(S)=N_G(S)$. We also let
$N_G^{\geq i}(S)=\bigcup_{j=i}^{\infty}N_G^j(S)$. When $S=\{s\}$, we write $N_G^i(s)$ instead of $N_G^i(\{s\})$. For a subgraph $H$ of $G$, we write
$N_G^i(H)$ for $N_G^i(V(H))$ and $N_G^{\geq i}(H)$ for
$N_G^{\geq i}(V(H))$. When $G$ is clear from the context, we omit the
subscript.

Let $S,T$ be disjoint subsets of $V(G)$. A vertex $v\in V(G)\setminus S$
is \emph{complete to} $S$ if $vs\in E(G)$ for every $s\in S$,
\emph{anticomplete to} $S$ if $vs\notin E(G)$ for every $s\in S$, and
\emph{mixed on} $S$ if it is neither complete nor anticomplete to $S$. We
say that $S$ is complete, respectively anticomplete, to $T$ if every vertex
of $S$ is complete, respectively anticomplete, to $T$. To describe an induced chair, we write $(c,uv,\{a,b\})$ where $c$ is its degree-three vertex, $cuv$ is the subdivided branch, and $a,b$ are the remaining leaves.

\section{Perfect divisibility and the proof of Theorem~\ref{p17}}

In this section, we prove Theorem~\ref{p17}, which we restate for
convenience.

\paleySeventeen*

\begin{figure}[ht]
    \centering
    \resizebox{0.35\textwidth}{!}{\begin{tikzpicture}[
    scale=1,
    Wvertex/.style={circle, draw=black, fill=white, scale=2},
    bvertex/.style={circle, draw=black, fill=black, scale=0.8},
    rvertex/.style={circle, draw=red, fill=red, scale=0.2}
]

% Vertices in cyclic order
\foreach \i in {0,...,16} {
    \pgfmathsetmacro{\ang}{90 - 360*\i/17}
    \node[
        bvertex,
        label={[font=\small]\ang:$\i$}
    ] (v\i) at (\ang:4) {};
}

% Paley graph P(17): quadratic residues mod 17 are ±1, ±2, ±4, ±8
\foreach \i in {0,...,16} {
    \foreach \d in {1,2,4,8} {
        \pgfmathtruncatemacro{\j}{mod(\i+\d,17)}
        \draw (v\i) -- (v\j);
    }
}

\end{tikzpicture}}
    \vskip -8pt
    \caption{The Paley graph on $17$ vertices.}
    \label{fig:p17}
    \vskip 2pt
\end{figure}
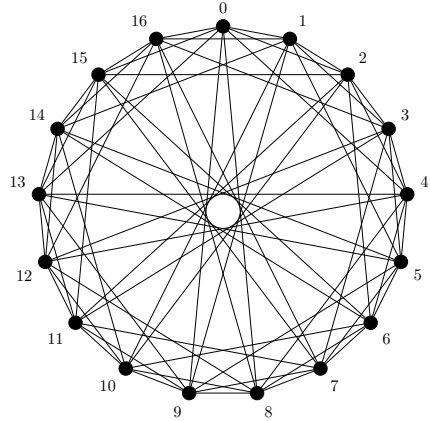

\begin{proof}[Proof of Theorem~\ref{p17}]
Let $G=P(17)$, and identify $V(G)$ with $\mathbb Z_{17}$. Distinct vertices
$x,y\in V(G)$ are adjacent if and only if $x-y$ is a quadratic residue
modulo $17$.

We first show that
$\chi(H)\leq\binom{\omega(H)+1}{2}$ for every induced subgraph $H$ of $G$.
The accompanying Python program\footnote{The program \texttt{verify\_paley17.py} is available in a fixed
version of the
\href{https://github.com/wzy3210/k-linear-divisibility-and-chair-free-binding-function/blob/25a130e409431f6ee4a9e97e6d7d216878d64488/verify_paley17.py}{GitHub repository}.} verifies that
$\omega(G)=\alpha(G)=3$, and $\chi(G)=6$. Hence
$\chi(G)=6=\binom{4}{2}=\binom{\omega(G)+1}{2}$.
Let $H$ be a proper induced subgraph of $G$. The assertion is immediate if
$H$ is empty or $\omega(H)=1$. If $\omega(H)=3$, then
$\chi(H)\leq\chi(G)=6=\binom{\omega(H)+1}{2}$. It remains to consider
$\omega(H)=2$. Then $H$ is triangle-free, so $N_H(v)$ is a stable set for
every $v\in V(H)$. Since $\alpha(H)\leq\alpha(G)=3$, we have
$\Delta(H)\leq3$. Applying Brooks' theorem to each component of $H$, and noting that the
possible exceptional odd cycles and complete graphs are also
$3$-colorable, we obtain $\chi(H)\leq3$. Therefore,
$\chi(H)\leq3=\binom{3}{2}=\binom{\omega(H)+1}{2}$.

It remains to show that $G$ is not perfectly divisible. Suppose, for a
contradiction, that $V(G)=A\sqcup B$, where $G[A]$ is perfect and
$\omega(G[B])<\omega(G)=3$. Thus, $G[B]$ is triangle-free. Since
$\alpha(G)=3$, every triangle-free induced subgraph of $G$ has at most $8$
vertices. Indeed, if $|B|\geq9$, then $R(3,4)=9$ implies that $G[B]$
contains either a triangle or a stable set of size $4$, a contradiction.
Hence $|B|\leq8$. Since $G[A]$ is perfect,
\[
|A|
\leq
\alpha(G[A])\chi(G[A])
=
\alpha(G[A])\omega(G[A])
\leq
3\cdot3
=
9.
\]
As $|A|+|B|=17$, we have $|A|=9$ and $|B|=8$.

The program enumerates all $51$ triangle-free $8$-vertex subsets of $G$
and verifies that they form two orbits under the affine automorphisms of
$G$, represented by
\[
B_1=\{0,1,3,6,7,10,12,13\}
\quad\text{and}\quad
B_2=\{0,1,3,4,6,10,11,15\}.
\]
Let $A_i=V(G)\setminus B_i$ for $i\in\{1,2\}$. The graph $G[A_1]$
contains the induced cycle
$2-4-5-14-15-2$,
while $G[A_2]$ contains the induced cycle
$5-7-8-12-14-5$.
Thus, neither $G[A_1]$ nor $G[A_2]$ is perfect. Since automorphisms preserve
the property of being perfect, $G[A]$ is not perfect, a contradiction. Therefore, $P(17)$ is
not perfectly divisible.

Finally, suppose that a graph $F$ contains $P(17)$ as an induced subgraph.
If $F$ were perfectly divisible, then every induced subgraph of $F$ would be
perfectly divisible, including the induced copy of $P(17)$, a contradiction.
\end{proof}

\begin{proposition}\label{prop:closure}
If $G$ and $H$ are $(k,\ell)$-linearly divisible, then $G\vee H$ is $(k,\ell)$-linearly divisible.
\end{proposition}

\begin{proof}
Let $G$ and $H$ be $(k,\ell)$-linearly divisible.
Let $F$ be a connected induced subgraph of $G\vee H$, and let
$F_G=F[V(F)\cap V(G)]$ and
$F_H=F[V(F)\cap V(H)]$.

If one of $F_G$ and $F_H$ is empty, then $F$ is a connected induced
subgraph of $G$ or $H$, and hence admits a $(k,\ell)$-linear division.
We may therefore assume that both are nonempty.

Let $\omega_1=\omega(F_G)$ and $\omega_2=\omega(F_H)$. 
Since $F_G$ is complete to $F_H$, we have $\omega(F)=\omega_1+\omega_2$.
Suppose first that $\omega_1,\omega_2\geq2$.
Let $D_1=F_G$ and $D_2=F_H$. Then $D_1$ and $D_2$ are complete to one another. Hence, we have $G_1=F$ and $G_2=\emptyset$ where $r=2$ and $\alpha=0$. 
Thus, this gives a $(k,\ell)$-linear division of $F$.
Next suppose that $\omega_1=1$ and $\omega_2\geq2$.
Let $D_1=G_1=F_H$ and $G_2=F_G.$
Then $\omega(G_1)=\omega_2=\omega(F)-1$ and $\chi(G_2)=1$. 
Thus, this also gives a $(k,\ell)$-linear division of $F$.
The case $\omega_2=1$ and $\omega_1\geq2$ is symmetric.
Finally, suppose that $\omega_1=\omega_2=1$.
Then both $F_G$ and $F_H$ are stable sets. Since $F_G$ is complete to $F_H$,  $F$ is a  complete bipartite graph, and hence $\chi(F)=\omega(F)=2.$
Taking $G_1=\emptyset$, $G_2=F,$ and $ \alpha=\min\{\ell,2\}$
gives the required division, since $\chi(G_2)=2
\leq
k\alpha\left(2-\frac{\alpha-1}{2}\right)$.

Therefore, $G\vee H$ is $(k,\ell)$-linearly divisible.
\end{proof}

\begin{proposition}\label{prop:p17-linear}
The Paley graph $P(17)$ is $(1,1)$-linearly divisible. Moreover, for every positive integer $t$,  $P(17)\vee K_t$ is $(1,1)$-linearly divisible. 
\end{proposition}

\begin{proof}
Let $H$ be a connected induced subgraph of $P(17)$, and write
$\omega=\omega(H)$. If $\omega=1$, then $r=0$ and $\alpha=1$ give the
required division. Suppose that $\omega=2$. By the proof of
Theorem~\ref{p17}, $H$ is $3$-colorable. If $H$ is bipartite, take $r=0$
and $\alpha=1$. Otherwise, choose a proper $3$-coloring of $H$, let $D_1$ be the
subgraph induced by one color class, and let $G_2=H-V(D_1)$. Then $G_2$ is bipartite, so
$G_1=D_1$, $r=1$, and $\alpha=1$ give a $(1,1)$-linear division.

Assume that $\omega=3$. Let $B_1$ and $A_1$ be as in the proof of
Theorem~\ref{p17}. The stable sets $\{2,8,14\}$, $\{4,9,15\}$, and
$\{5,11,16\}$ partition $A_1$, so $P(17)[A_1]$ is $3$-colorable. Choose an
edge $uv$ of $H$. The affine automorphisms of $P(17)$ act transitively on
its edges, so there is an automorphism $\varphi$ such that
$u,v\in\varphi(B_1)$. Let $B=V(H)\cap\varphi(B_1)$ and
$A=V(H)\cap\varphi(A_1)$. Then $H[B]$ is triangle-free and contains
$uv$, while $H[A]$ is $3$-colorable. Thus, taking $G_1=D_1=H[B]$,
$G_2=H[A]$, $r=1$, and $\alpha=1$ gives a $(1,1)$-linear division.

Since $K_t$ and $P(17)$ are $(1,1)$-linearly divisible, by Proposition~\ref{prop:closure}, $P(17)\vee K_t$ is $(1,1)$-linearly divisible. 
\end{proof}

\section{\texorpdfstring{$(k,\ell)$}{(k,l)}-linear divisibility and the proof of Theorem~\ref{thm:kld}}

We now prove Theorem~\ref{thm:kld}, which we restate for convenience.

\linearlydivisiblebound*

\begin{proof}
We proceed by induction on $\omega$. For each value of $\omega$, it
suffices to prove the result when $G$ is connected, since the result for an
arbitrary graph then follows by applying the connected case to its
components. Let $G$ be connected and write $\omega=\omega(G)$. Define
$f(s)=k\binom{s+1}{2}=\frac{ks(s+1)}{2}$. The result is immediate if
$\omega\leq1$, so assume that $\omega\geq2$. Every induced subgraph of
a $(k,\ell)$-linearly divisible graph is again $(k,\ell)$-linearly divisible.

Apply the definition to $G$, and let
$\alpha,r,G_1,G_2,D_1,\ldots,D_r$ be as in the definition. Suppose first
that $r=0$. Then $G_1$ is empty and $G_2=G$. By condition~(2),
$\chi(G)\leq k\alpha\left(\omega-\frac{\alpha-1}{2}\right)
=f(\omega)-f(\omega-\alpha)\leq f(\omega)$.

Suppose next that $r=1$. Then $G_1=D_1$ and
$\omega(D_1)\leq\omega-\alpha<\omega$. By induction,
$\chi(G_1)\leq f(\omega-\alpha)$, while condition~(2) gives
$\chi(G_2)\leq f(\omega)-f(\omega-\alpha)$. Hence
$\chi(G)\leq f(\omega)$.

Assume that $r\geq2$, and write $\omega_i=\omega(D_i)\geq2$ for
$i\in[r]$. Since $D_1,\ldots,D_r$ are pairwise complete,
$\sum_{i=1}^r\omega_i=\omega(G_1)\leq\omega-\alpha$. Each $D_i$ has
clique number smaller than $\omega$, so induction gives
\[
\chi(G_1)
\leq
\sum_{i=1}^r\chi(D_i)
\leq
\sum_{i=1}^rf(\omega_i).
\]
Since $f$ is increasing and convex, subject to $\omega_i\geq2$ and
$\sum_i\omega_i\leq\omega-\alpha$, the final sum is maximized by taking
$r-1$ of the $\omega_i$ equal to $2$. Therefore,
\[
\chi(G_1)
\leq
f(\omega-\alpha-2r+2)+(r-1)f(2).
\]
Let $t=\alpha+2r-2$. Then $\omega-\alpha-2r+2=\omega-t$ and
$f(2)=3k$, so $\chi(G_1)\leq f(\omega-t)+3k(r-1)$. Condition~(2) also
gives
\[
\chi(G_2)
\leq
k\left(t\left(\omega-\frac{t+2}{2}\right)+\frac{3\alpha}{2}\right).
\]
Consequently,
\begin{align*}
\chi(G)
&\leq \chi(G_1)+\chi(G_2)\\
&\leq f(\omega-t)+3k(r-1)
+k\left(t\left(\omega-\frac{t+2}{2}\right)+\frac{3\alpha}{2}\right)\\
&=f(\omega)-\frac{3kt}{2}+3k(r-1)+\frac{3k\alpha}{2}\\
&=f(\omega),
\end{align*}
where the last equality follows from $t=\alpha+2r-2$.
\end{proof}

The following simple consequence of the definition will be used in the
proof of Theorem~\ref{main thm}.

\begin{corollary}\label{cor:two-blocks}
Let $G$ be a connected graph with $\omega(G)=\omega$, and suppose that
every proper induced subgraph of $G$ is $(2,2)$-linearly divisible.
Suppose that $V(G)=V(D_1)\sqcup V(D_2)\sqcup V(G_2)$, where
$D_1,D_2$, and $G_2$ are induced subgraphs of $G$, the graphs $D_1$ and
$D_2$ are nonempty and complete to one another, and
$\chi(G_2)\leq2\omega-1$. Then $G$ is $(2,2)$-linearly divisible.
\end{corollary}

\begin{proof}
By the hypothesis on proper induced subgraphs, it suffices to give a
$(2,2)$-linear division of $G$. Let $\omega_i=\omega(D_i)$ for
$i\in\{1,2\}$. Suppose first that $\omega_1,\omega_2\geq2$. Since
$D_1$ and $D_2$ are complete to one another,
$\omega\geq\omega_1+\omega_2\geq4$, and hence
$2\omega-1\leq4(\omega-2)$. Thus, taking
$G_1=G[V(D_1)\cup V(D_2)]$, $r=2$, and $\alpha=0$ gives a
$(2,2)$-linear division of $G$.

Suppose now, without loss of generality, that $\omega_1=1$. If
$\omega_2\geq2$, let $G_1=D_2$ and
$\widetilde G_2=G[V(D_1)\cup V(G_2)]$; otherwise, let $G_1=D_1$ and
$\widetilde G_2=G[V(D_2)\cup V(G_2)]$. In either case,
$\omega(G_1)\leq\omega-1$ and
$\chi(\widetilde G_2)\leq\chi(G_2)+1\leq2\omega$. Thus, $r=1$ and
$\alpha=1$ give a $(2,2)$-linear division of $G$.
\end{proof}

\section{A variant of the template method}\label{section:template}

We now introduce an adapted version of the template method and prove several
preliminary lemmas for Theorem~\ref{main thm}. Let $G$ be a chair-free graph
with clique number $\omega=\omega(G)$.
A \emph{$q$-template} in $G$ is an induced complete $q$-partite graph $Q=G[V_1\cup\cdots\cup V_q]$ such that $|V_q|=2$ and $|V_i|=1$ for every $i\in[q-1]$. Let $V_q=\{a_1,a_2\}$, $V_i=\{v_i\}$ for $i\in[q-1]$, and
$K=V(Q)\setminus V_q$.
A $q$-template $Q$ is \emph{maximal} if the following conditions hold.
\begin{enumerate}[(E1)]
\item No vertex in $V(G)\setminus V(Q)$ is complete to $V(Q)$.
\item There do not exist nonadjacent vertices $c,d\in V(G)\setminus V(Q)$
and an index $i\in\{1,2\}$ such that $\{c,d\}$ is complete to $K$ and both
$c$ and $d$ are adjacent to $a_i$ and nonadjacent to $a_{3-i}$.
\end{enumerate}

We first give a simple extension property.

\begin{lemma}\label{lem:template-extension}
Let $Q_0$ be a $q$-template in a graph $G$, with last part
$\{a_1,a_2\}$, and let $K_0=V(Q_0)\setminus\{a_1,a_2\}$. Then there is a
maximal $q'$-template $Q$ such that $K_0\subseteq V(Q)$ and
$\{a_1,a_2\}\cap V(Q)\neq\emptyset$. Moreover, every vertex of $K_0$
remains in a singleton part of $Q$.
\end{lemma}

\begin{proof}
Starting with $Q_0$, repeat the following operations. If (E1) fails for a
vertex $z$, add $\{z\}$ as a singleton part. If (E2) fails for nonadjacent
vertices $c,d$ and $a_i$, add $\{a_i\}$ as a singleton part and replace the
last part by $\{c,d\}$. After each operation, relabel the two vertices in the
last part as $a_1,a_2$. Each operation produces a template with one more part
and preserves every existing singleton part. The first time the second
operation is used, one of the two vertices in the original last part becomes
a singleton and is preserved thereafter. Since every $q$-template contains a
clique of size $q$, the process terminates with the required maximal template.
\end{proof}

Given a maximal $q$-template $Q$, partition $N(Q)$ as follows:
\begin{align*}
P&=\{v\in N(Q):v\text{ is mixed on }V_q\},\\
Z&=\{v\in N(Q):v\text{ is complete to }V_q\},\\
W&=\{v\in N(Q):v\text{ is anticomplete to }V_q\}.
\end{align*}
For each $I\subseteq K$, let
\[
W_I
=
\{v\in W:v\text{ is complete to }I\text{ and anticomplete to }K\setminus I\}.
\]
For a component $X$ of $G[W]$, let
\[
Z_X=\{z\in Z:z\text{ is complete to }V(X)\}.
\]
Let $X_0$ denote the union of the clique components of $G[W]$. Observe that
$N(Q)=P\sqcup Z\sqcup W$. Every component of $G[W\setminus X_0]$ is
non-clique and hence contains two nonadjacent vertices. Liu, Schroeder, Wang,
and Yu~\cite[Lemma~2.2]{LSWY2023} proved the following structural lemma; see
Figure~\ref{fig:tem}.

\begin{figure}[htb!]
    \centering
    \resizebox{0.65\textwidth}{!}{\input{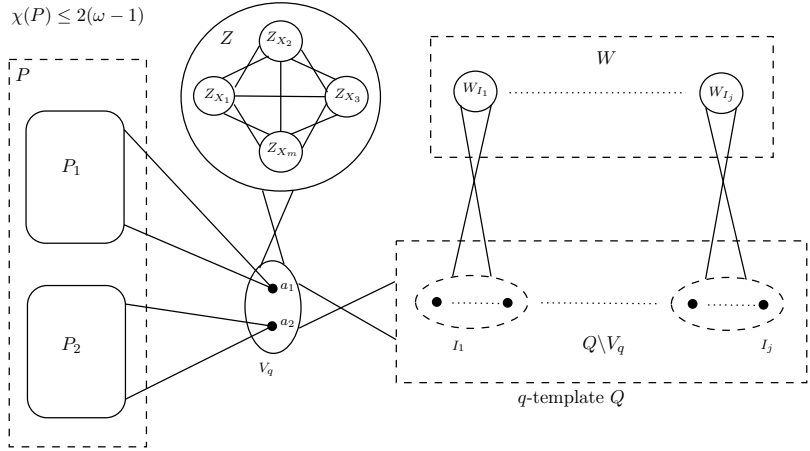}}
    \vskip -2pt
    \caption{The structure of $N[Q]$.}
    \label{fig:tem}
\end{figure}

\begin{lemma}[Liu--Schroeder--Wang--Yu]
\label{lem:chair_structure}
Let $G$ be a chair-free graph, let $Q$ be a maximal template in $G$, and
let $P,Z,W$ be defined as above. Then:
\begin{enumerate}[(1)]
\item $W_I$ is anticomplete to $W_{I'}$ for distinct subsets $I,I'$ of $K$;

\item for every $z\in Z$ and every component $X$ of $G[W]$, $z$ is either complete or anticomplete to $V(X)$;

\item for every component $X$ of $G[W\setminus X_0]$, $Z_X$ is
complete to $Z\setminus Z_X$.
\end{enumerate}
\end{lemma}

By Lemma~\ref{lem:chair_structure}(1), every component of $G[W]$ is
contained in a unique set $W_I$ for some $I\subseteq K$.
We also need the following consequences of chair-freeness.

\begin{lemma}\label{lm:W}
\begin{enumerate}[(1)]
\item $W$ is anticomplete to $N_G^2(Q)$.

\item Let $X$ be a component of $G[W\setminus X_0]$ contained in $W_I$.
Then $Z\setminus Z_X$ is complete to $I$, and is
anticomplete to $N_G^2(Q)$.
\end{enumerate}
\end{lemma}

\begin{proof}
(1) Suppose that $wd\in E(G)$ for some $w\in W$ and $d\in N_G^2(Q)$.
Since $w\in N(Q)$ and is anticomplete to $V_q$, it has a neighbor
$v_i\in K$. Then $(v_i,wd,\{a_1,a_2\})$ is an induced chair, a
contradiction.

(2) Note that $I\neq\emptyset$, since vertices of $X$ belong to
$N(Q)$ and are anticomplete to $V_q$. Suppose that some $z\in Z\setminus
Z_X$ is nonadjacent to a vertex $v_i\in I$. By
Lemma~\ref{lem:chair_structure}(2), $z$ is anticomplete to $X$.
Since $X$ is a connected non-clique graph, it contains nonadjacent vertices
$x_1,x_2$. Then $(v_i,a_1z,\{x_1,x_2\})$ is an induced chair, a
contradiction.

Suppose that $zd\in E(G)$ for some $z\in Z\setminus Z_X$ and
$d\in N_G^2(Q)$. Choose $v_i\in I$ and nonadjacent vertices
$x_1,x_2\in V(X)$. By (1), $d$ is anticomplete to $X$, and by the preceding paragraph,
$zv_i\in E(G)$. Thus, $(v_i,zd,\{x_1,x_2\})$ is an induced chair, a
contradiction. Hence, (2) holds.
\end{proof}

The next lemma, which shows that $P$ is the union of two perfect graphs, is
similar to Lemma~2.5 of~\cite{LSWY2023}. We include the proof for
completeness.

\begin{lemma}\label{pperfect}
Let $Q$ be a maximal $q$-template in $G$. For each $i\in\{1,2\}$, let
$P_i=P\cap N(a_i)$. Then $G[P_i]$ is perfect and
$\chi(G[P])\leq2(\omega-1)$.
\end{lemma}

\begin{proof}
By symmetry, it suffices to prove that $G[P_1]$ is perfect. For each
$u\in P_1$, let $j_u$ be the smallest index $j\in[q-1]$ such that
$uv_j\notin E(G)$, if such an index exists, and let $j_u=q$ otherwise. For
$j\in[q]$, let $P_1^j=\{u\in P_1:j_u=j\}$.

Each $P_1^j$ is a clique. Suppose first that $j<q$ and that nonadjacent
vertices $c,d$ belong to $P_1^j$. Then $(a_1,v_ja_2,\{c,d\})$ is an
induced chair, a contradiction. If $P_1^q$ contains nonadjacent vertices
$c,d$, then these vertices, together with the singleton parts of $Q$ and the
singleton part $\{a_1\}$, form a $(q+1)$-template, contradicting the
maximality of $Q$.

Suppose that $G[P_1]$ contains an odd hole
$u_0,u_1,\ldots,u_\ell$, where $\ell\geq4$ is even and subscripts are taken
modulo $\ell+1$. Choose $u_0$ so that $j_{u_0}$ is minimum, and let
$j=j_{u_0}$. If $j=q$, then the hole is contained in the clique $P_1^q$, a
contradiction. Thus, $j<q$. We have $j_{u_1},j_{u_\ell}\geq j$, and since
$u_1u_\ell\notin E(G)$, these two indices are not both equal to $j$.
After reversing the hole if necessary, assume that $j_{u_1}>j$. Since
$u_0$ is nonadjacent to $u_i$ for $2\leq i\leq\ell-1$, we also have
$j_{u_i}>j$ for these indices. Hence $v_j$ is adjacent to $u_1$ and
$u_{\ell-1}$ and nonadjacent to $u_0$. Then
$(v_j,u_1u_0,\{u_{\ell-1},a_2\})$ is an induced chair, a contradiction.

Now suppose that $G[P_1]$ contains an odd antihole
$u_0,u_1,\ldots,u_\ell$, where $\ell\geq6$ is even and subscripts are taken
modulo $\ell+1$. Again choose $u_0$ so that $j_{u_0}$ is minimum, and let
$j=j_{u_0}$. If $j=q$, then the antihole is contained in the clique $P_1^q$,
a contradiction. Thus, $j<q$. We claim that if $v_ju_i\notin E(G)$, then
$v_ju_{i+2}\notin E(G)$. Indeed, $v_ju_i\notin E(G)$ implies
$j_{u_i}=j$. Since $P_1^j$ is a clique and $u_iu_{i+1}\notin E(G)$, we
have $j_{u_{i+1}}>j$, so $v_ju_{i+1}\in E(G)$. If
$v_ju_{i+2}\in E(G)$, then $(v_j,u_{i+2}u_i,\{u_{i+1},a_2\})$ is an
induced chair, a contradiction. The claim follows.

Starting from $v_ju_0\notin E(G)$ and repeatedly applying the claim, we conclude that $v_j$ is anticomplete to the antihole. Thus $j_{u_i}\leq j$ for every $i$, and the minimality of $j$ gives $j_{u_i}=j$ for every $i$.
This places the antihole in the clique $P_1^j$, a contradiction.

By the Strong Perfect Graph Theorem, $G[P_1]$ is perfect. By symmetry,
$G[P_2]$ is also perfect. Since $a_i$ is complete to $P_i$, we have
$\omega(G[P_i])\leq\omega-1$ for $i\in\{1,2\}$. Therefore,
$\chi(G[P])
\leq
\chi(G[P_1])+\chi(G[P_2])
\leq
2(\omega-1)$.
\end{proof}

We are now ready for the main template lemma.

\begin{lemma}\label{lem:template-divisible}
Let $G$ be a connected chair-free graph with clique number
$\omega=\omega(G)\geq3$, and let $Q$ be a maximal $q$-template in $G$, where
$q\geq2$. Suppose that every proper induced subgraph of $G$ is
$(2,2)$-linearly divisible, $N_G^{\geq3}(Q)=\emptyset$, and
$\chi(G[N_G^2(Q)])\leq\omega$. Then $G$ is $(2,2)$-linearly divisible.
\end{lemma}

\begin{proof}
Let $P,Z,W,X_0$ be defined as above. Since
$N_G^{\geq3}(Q)=\emptyset$,
$$V(G)=V(Q)\sqcup P\sqcup Z\sqcup W\sqcup N_G^2(Q).$$
Let $L=P\cup V(Q)\cup X_0\cup N_G^2(Q)$.

\begin{claim}\label{cl:L}
$\chi(G[V(Q)\cup X_0\cup N_G^2(Q)])\leq\omega$ and
$\chi(G[L])\leq3\omega-2$.
\end{claim}

\begin{proof}
By the definition of $N_G^2(Q)$ and Lemma~\ref{lm:W}(1),
$N_G^2(Q)$ is anticomplete to $V(Q)\cup X_0$. Since
$\chi(G[N_G^2(Q)])\leq\omega$, it is enough to color
$G[V(Q)\cup X_0]$ with $\omega$ colors.

Fix a palette of $\omega$ colors. Since $q\leq\omega$, color $Q$ by giving
the vertices of $K$ distinct colors and giving $a_1,a_2$ one additional
color. Let $X$ be a component of $G[X_0]$ contained in $W_I$. Since $X$
is a clique complete to $I$, we have $|X|+|I|\leq\omega$. Moreover, $X$
is anticomplete to $V(Q)\setminus I$. We may therefore color $X$ with
$|X|$ distinct colors not used on $I$. Distinct components of $G[X_0]$
are anticomplete, so these colors may be reused. Hence,
$\chi(G[V(Q)\cup X_0\cup N_G^2(Q)])\leq\omega$. Lemma~\ref{pperfect}
now gives
\[
\chi(G[L])
\leq
\chi(G[P])+\chi(G[V(Q)\cup X_0\cup N_G^2(Q)])
\leq
3\omega-2.\qedhere
\]
\end{proof}

Suppose first that $W=X_0$, and let
$R=G[\{a_2\}\cup P_2\cup W\cup N_G^2(Q)]$. By Claim~\ref{cl:L},
$\chi(G[\{a_2\}\cup W\cup N_G^2(Q)])\leq\omega$, while
$\chi(G[P_2])\leq\omega-1$. Thus, $\chi(R)\leq2\omega-1$. Moreover,
$G[N(a_1)]$ is nonempty and has clique number at most $\omega-1$, while
$a_1$ is anticomplete to $R$. Hence, taking $G_1=D_1=G[N(a_1)]$,
$G_2=G[V(R)\cup\{a_1\}]$, $r=1$, and $\alpha=1$ gives a
$(2,2)$-linear division.

Assume that $W\neq X_0$, and let $X_1,\ldots,X_s$ be the non-clique
components of $G[W]$. For each $i\in[s]$, let $I_i\subseteq K$ satisfy
$V(X_i)\subseteq W_{I_i}$, and let
$Z_i=\{z\in Z:z\text{ is complete to }X_i\}$.

\begin{claim}\label{IZ}
Let $F$ be a component of $G[W]$ contained in $W_I$. Then $I$ and $I_i$
are comparable under inclusion for every $i\in[s]$. Moreover,
$Z_1,\ldots,Z_s$ form a chain under inclusion.
\end{claim}

\begin{proof}
Fix $i\in[s]$. The first assertion is immediate if $F=X_i$. Otherwise,
if $I$ and $I_i$ are incomparable, choose $c\in I_i\setminus I$,
$d\in I\setminus I_i$, nonadjacent vertices $x_1,x_2\in V(X_i)$, and
$y\in V(F)$. Then $(c,dy,\{x_1,x_2\})$ is an induced chair, a
contradiction.

Now let $i,j\in[s]$ be distinct. If $Z_i$ and $Z_j$ are incomparable,
choose $z\in Z_i\setminus Z_j$, $z'\in Z_j\setminus Z_i$, nonadjacent
vertices $x_1,x_2\in V(X_i)$, and $y\in V(X_j)$. By
Lemma~\ref{lem:chair_structure}(2)--(3), the vertices $(z,z'y,\{x_1,x_2\})$ induce a
chair, a contradiction.
\end{proof}

Applying the first assertion of Claim~\ref{IZ} with $F=X_j$ shows that
$I_1,\ldots,I_s$ form a chain. Each $I_i$ is nonempty because
$X_i\subseteq W\subseteq N(Q)$ and $X_i$ is anticomplete to $V_q$. Choose
$J\in\{I_1,\ldots,I_s\}$ such that $J\subseteq I_i$ for every $i\in[s]$.
After relabeling, assume that
$Z_1\subseteq Z_2\subseteq\cdots\subseteq Z_s$. Let
\[
\begin{aligned}
D_1&=G[Z_1\cup J],\\
D_2&=G[(Z\setminus Z_1)\cup(W\setminus X_0)],\\
G_2&=G[P\cup(V(Q)\setminus J)\cup X_0\cup N_G^2(Q)].
\end{aligned}
\]
Then $V(G)=V(D_1)\sqcup V(D_2)\sqcup V(G_2)$. Note that $D_1$ and $D_2$
are nonempty and complete to one another. Indeed, $Z_1$ is complete to
$Z\setminus Z_1$ by Lemma~\ref{lem:chair_structure}(3), and
$Z_1\subseteq Z_i$ for every $i\in[s]$, so $Z_1$ is complete to
$W\setminus X_0$. The choice of $J$ shows that $J$ is complete to
$W\setminus X_0$, while Lemma~\ref{lm:W}(2), applied to $X_1$, shows that
$Z\setminus Z_1$ is complete to $J$. Moreover, $D_2$ is anticomplete to
$N_G^2(Q)$ by Lemma~\ref{lm:W}(1)--(2), and $\omega(D_2)\geq2$ because
$W\setminus X_0$ contains the connected non-clique graph $X_1$. By
Claim~\ref{cl:L}, $\chi(G_2)\leq3\omega-2$.

Suppose first that $\omega(D_1)\geq2$. Since $D_1$ and $D_2$ are complete
to one another, $\omega(D_1)+\omega(D_2)\leq\omega$. If $\omega\geq6$,
then $3\omega-2\leq4\omega-8$, so taking
$G_1=G[V(D_1)\cup V(D_2)]$, $r=2$, and $\alpha=0$ gives a
$(2,2)$-linear division.

We may therefore assume that $\omega\in\{4,5\}$. Choose
$i\in\{1,2\}$ so that $\omega(D_i)$ is maximum. Then
$\omega(D_i)\leq\omega-2$ and $\omega(D_{3-i})=2$. Since $D_{3-i}$ is
$3$-colorable~\cite{Randerath-Schiermeyer2002}, the graph induced by
$V(D_{3-i})\cup V(G_2)$ has chromatic number at most
$3+(3\omega-2)=3\omega+1\leq4\omega-2$. Thus, taking $G_1=D_i$,
$r=1$, and $\alpha=2$ gives a $(2,2)$-linear division.

We may therefore assume that $\omega(D_1)=1$. Since $J$ is a nonempty
clique contained in $V(D_1)$, we have $J=\{v_j\}$ for some $j\in[q-1]$.
The vertex $v_j$ is complete to $X_0$. Indeed, let $F$ be a component of
$G[X_0]$ contained in $W_I$. The set $I$ is nonempty and is comparable
with $J$ by Claim~\ref{IZ}; hence $v_j\in I$, so $v_j$ is complete to
$F$.

Let $H=G[V(D_2)\cup(V(Q)\setminus\{v_j\})\cup X_0]$. The vertex $v_j$
is complete to $H$, so $\omega(H)\leq\omega-1$. Moreover, $H$ is
anticomplete to $N_G^2(Q)$ by the definition of $N_G^2(Q)$ and
Lemma~\ref{lm:W}(1)--(2). Choose a stable set $S\subseteq N_G^2(Q)$ such
that $\omega(G[N_G^2(Q)\setminus S])\leq\omega-1$. Such a set exists:
take $S=\emptyset$ if $\omega(G[N_G^2(Q)])\leq\omega-1$. Otherwise, fix
an $\omega$-coloring of $G[N_G^2(Q)]$ and take one color class as $S$.
Every $\omega$-clique uses all $\omega$ colors, so
$\omega(G[N_G^2(Q)\setminus S])\leq\omega-1$.

Let $D=G[V(H)\cup(N_G^2(Q)\setminus S)]$ and
$\widetilde G_2=G[V(D_1)\cup P\cup S]$. Since $H$ is anticomplete to
$N_G^2(Q)$, the graph $D$ is nonempty and satisfies
$\omega(D)\leq\omega-1$. Moreover,
$\chi(\widetilde G_2)\leq1+(2\omega-2)+1=2\omega$. Hence, taking
$G_1=D$, $r=1$, and $\alpha=1$ gives a $(2,2)$-linear division.
\end{proof}

\section{Structure of chair-free graphs}\label{section:structure}

In this section, we establish several structural properties of chair-free
graphs that will be used in the proof of Theorem~\ref{main thm}. Throughout
this section, let $G$ be a connected chair-free graph with clique number
$\omega(G)=\omega$.

\medskip
\noindent\textbf{Global setup.}
Let $v_0$ be a vertex contained in a maximum clique of $G$. Then
$\omega(G[N(v_0)])=\omega-1$. Assume that $G-N[v_0]\neq\emptyset$, and let
$H_1,\ldots,H_t$ be the components of $G-N[v_0]$, where $H_1$ has maximum
chromatic number. When $H_1$ is not a clique, relabel the components so
that, for some $t'\in[t]$, $H_1,\ldots,H_{t'}$ are not cliques
and $H_{t'+1},\ldots,H_t$ are cliques.

\begin{figure}[htb!]
    \centering
    \begin{minipage}{0.42\textwidth}
        \centering
        \resizebox{0.75\textwidth}{!}{\begin{tikzpicture}[
    line width=0.5pt,
    comp/.style={
        draw,
        rounded corners=2pt,
        minimum width=8mm,
        minimum height=13mm,
        inner sep=2pt
    },
    neighborhood/.style={
        draw,
        ellipse,
        minimum width=3.8cm,
        minimum height=1.25cm
    },
    vertex/.style={
        circle,
        fill,
        inner sep=1.3pt
    }
]
    \node[neighborhood] (Nv) at (0,0) {};
    \node at (0,0) {$N(v_0)$};

    \node[vertex,label=below:$v_0$] (v) at (0,-1.05) {};

    \node[comp] (H1) at (-2.7,1.85) {$H_1$};
    \node[comp] (H2) at (-0.9,1.85) {$H_2$};
    \node at (0.9,1.85) {$\cdots$};
    \node[comp] (Ht) at (2.7,1.85) {$H_t$};

    \draw[dashed] (H2.east) -- (Ht.west);

    \draw (H1.south) -- (Nv.145);
    \draw (H2.south) -- (Nv.110);
    \draw (Ht.south) -- (Nv.35);

    \draw[densely dotted] (Nv.225) -- (v);
    \draw[densely dotted] (Nv.315) -- (v);
\end{tikzpicture}}
    \end{minipage}
    \hfill
    \begin{minipage}{0.52\textwidth}
        \centering
        \resizebox{0.7\textwidth}{!}{\begin{tikzpicture}[
    line width=0.5pt,
    yset/.style={
        draw,
        minimum width=10mm,
        minimum height=7mm,
        inner sep=1pt
    },
    rset/.style={
        draw,
        ellipse,
        minimum width=12mm,
        minimum height=10mm,
        inner sep=1pt
    },
    nregion/.style={
        draw,
        ellipse,
        minimum width=3.2cm,
        minimum height=1.15cm
    }
]
    \coordinate (xpos) at (0,0);
    \coordinate (vpos) at (0,-1.05);

    % Neighborhood of v
    \node[nregion] (Nv) at (xpos) {};
    \node at (2.2,0.12) {$N(v_0)$};

    % The sets Y_0,Y_1,\ldots,Y_m
    \node[yset] (Y0)  at (-3.3,1.45) {$Y_0$};
    \node[yset] (Y1)  at (-1.8,1.45) {$Y_1$};
    \node at (0,1.45) {$\cdots$};
    \node[yset] (Ym1) at (1.8,1.45) {$Y_{m-1}$};
    \node[yset] (Ym)  at (3.3,1.45) {$Y_m$};

    % The components R_1,\ldots,R_m
    \node[rset] (R1)  at (-1.8,2.55) {$R_1$};
    \node at (0,2.55) {$\cdots$};
    \node[rset] (Rm1) at (1.8,2.55) {$R_{m-1}$};
    \node[rset] (Rm)  at (3.3,2.55) {$R_m$};

    % Adjacencies between R_i and Y_i
    \draw (R1.south)  -- (Y1.north);
    \draw (Rm1.south) -- (Ym1.north);
    \draw (Rm.south)  -- (Ym.north);

    % Pairwise completeness, represented by crosses
    \draw (Y0.north east) -- (Y1.south west);
    \draw (Y0.south east) -- (Y1.north west);

    \draw (Ym1.north east) -- (Ym.south west);
    \draw (Ym1.south east) -- (Ym.north west);

    % x is complete to the sets Y_i
    \draw (Y0.south)  -- (xpos);
    \draw (Y1.south)  -- (xpos);
    \draw (Ym1.south) -- (xpos);
    \draw (Ym.south)  -- (xpos);

    % Adjacencies between v and N(v)
    \draw[dashed] (vpos) -- (Nv.235);
    \draw[dashed] (vpos) -- (Nv.305);

    % Vertices x and v
    \fill (xpos) circle (1.4pt);
    \node[below=1pt] at (xpos) {$x$};

    \fill (vpos) circle (1.4pt);
    \node[below=1pt] at (vpos) {$v_0$};
\end{tikzpicture}}
    \end{minipage}
    \vskip -8pt
    \caption{(a) The global structure of $G$ and (b) the structure of
    $G[N[v_0]\cup V(H_1)]$.}
    \label{fig:chairfree}
    \vskip 2pt
\end{figure}
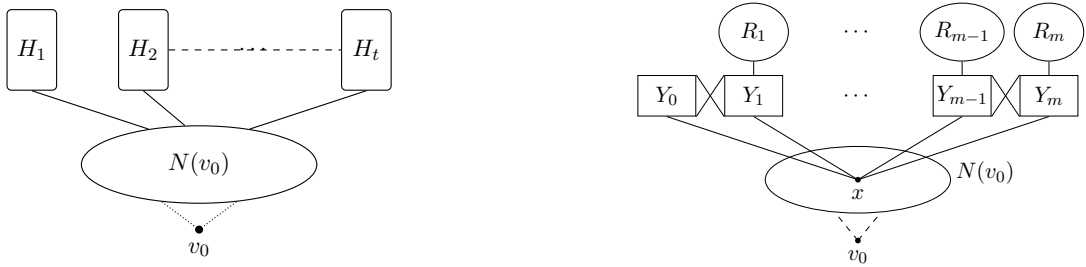

Since $G$ is connected, some vertex of $N(v_0)$ has a neighbor in $H_1$.
For each such vertex $x$, let $m(x)$ be the number of components of
$H_1-N(x)$. Choose $x\in N(v_0)$ so that $m(x)>0$ whenever possible. Let
$Y=N(x)\cap V(H_1)$, and let $R_1,\ldots,R_m$ be the components of $H_1-Y$,
where $m=m(x)$. For each $i\in[m]$, let
$Y_i=\{y\in Y:N(y)\cap V(R_i)\neq\emptyset\}$,
and
$Y_0=Y\setminus\bigcup_{i=1}^mY_i$.
Each $Y_i$ with $i\in[m]$ is nonempty because $H_1$ is connected. 
Figure~\ref{fig:chairfree} illustrates this decomposition.

\begin{lemma}\label{t>1}
If $t'\geq2$, $m\geq1$, and $Y$ is not a clique, then some vertex
$w\in N(v_0)$ is complete to $\bigcup_{i=1}^{t'}V(H_i)$.
\end{lemma}

\begin{proof}
Choose nonadjacent vertices $y_1,y_2\in Y$. We first show that $x$ is
anticomplete to $H_i$ for every $i\in\{2,\ldots,t\}$. Otherwise, let
$a_i\in V(H_i)$ be adjacent to $x$. Since $m\geq1$, there are adjacent
vertices $y\in Y$ and $r\in V(R_1)$. Then
$(x,yr,\{v_0,a_i\})$ is an induced chair, a contradiction.

Fix $i\in\{2,\ldots,t\}$. Since $G$ is connected, there are vertices
$w_i\in N(v_0)$ and $a_i\in V(H_i)$ with $w_ia_i\in E(G)$. If $w_i$ is
mixed on $H_1$, then connectedness of $H_1$ gives adjacent vertices
$b_1,b_2\in V(H_1)$ such that $w_ib_1\in E(G)$ and
$w_ib_2\notin E(G)$. Now $(w_i,b_1b_2,\{v_0,a_i\})$ is an induced
chair, a contradiction. Thus, $w_i$ is either complete or anticomplete to
$H_1$.

Suppose that $w_i$ is anticomplete to $H_1$. If $w_ix\notin E(G)$, then
$(x,v_0w_i,\{y_1,y_2\})$ is an induced chair; if $w_ix\in E(G)$, then
$(x,w_ia_i,\{y_1,y_2\})$ is an induced chair. Both cases are
contradictions, so $w_i$ is complete to $H_1$.

More generally, if a vertex $w\in N(v_0)$ is complete to $H_1$ and has a
neighbor in $H_i$, then $w$ is complete to $H_i$. Otherwise, connectedness
of $H_i$ gives adjacent vertices $h_1,h_2\in V(H_i)$ such that
$wh_1\in E(G)$ and $wh_2\notin E(G)$, and
$(w,h_1h_2,\{y_1,y_2\})$ is an induced chair, a contradiction. In
particular, each $w_i$ is complete to $H_1\cup H_i$.

For each vertex $w\in N(v_0)$ complete to $H_1$, let
$I_w=\{i\in[t']:w\text{ is complete to }H_i\}$. Choose such a vertex $w$
with $|I_w|$ maximum. Suppose that $I_w\neq[t']$, and choose
$i\in[t']\setminus I_w$. Since $1\in I_w$, we have $i\geq2$.
$w_i$ is complete to $H_1\cup H_i$, so $i\in I_{w_i}$. By maximality,
there is an index $j\in I_w\setminus I_{w_i}$.

Thus, $w$ is anticomplete to $H_i$, while $w_i$ is anticomplete to $H_j$.
Choose nonadjacent vertices $h_{i1},h_{i2}\in V(H_i)$ and a vertex
$h_j\in V(H_j)$. If $ww_i\notin E(G)$, then
$(w_i,v_0w,\{h_{i1},h_{i2}\})$ is an induced chair. If $ww_i\in E(G)$,
then $(w_i,wh_j,\{h_{i1},h_{i2}\})$ is an induced chair. Both cases are
contradictions. Therefore, $I_w=[t']$, and $w$ is complete to
$\bigcup_{i=1}^{t'}V(H_i)$.
\end{proof}

\begin{lemma}\label{lem:Y-pairwise-complete}
The sets $Y_0,Y_1,\ldots,Y_m$ are pairwise disjoint and pairwise complete.
Moreover, if $Y_i$ is not complete to $R_i$ for some $i\in[m]$, then $Y_j$
is a clique for every $j\in\{0,\ldots,m\}\setminus\{i\}$.
\end{lemma}

\begin{proof}
Suppose that $y\in Y_i\cap Y_j$ for some $1\leq i<j\leq m$. Choose
$r_i\in V(R_i)$ and $r_j\in V(R_j)$ adjacent to $y$. Then
$(y,xv_0,\{r_i,r_j\})$ is an induced chair, a contradiction. Hence
$Y_1,\ldots,Y_m$ are pairwise disjoint, and $Y_0$ is disjoint from each of them by definition.

Suppose next that $y_i\in Y_i$ and $y_j\in Y_j$ are nonadjacent for some
$0\leq i<j\leq m$. Choose $\ell\in\{i,j\}\setminus\{0\}$, let $y_s$ be
the other vertex, and choose $r\in V(R_\ell)$ adjacent to $y_\ell$. Since
$y_s\notin Y_\ell$, it is nonadjacent to $r$. Thus
$(x,y_\ell r,\{y_s,v_0\})$ is an induced chair, a contradiction.
Therefore, $Y_0,Y_1,\ldots,Y_m$ are pairwise complete.

Finally, suppose that $Y_i$ is not complete to $R_i$. Choose $y\in Y_i$
and adjacent vertices $r_1,r_2\in V(R_i)$ such that $yr_1\in E(G)$ and
$yr_2\notin E(G)$. If $Y_j$ is not a clique for some
$j\in\{0,\ldots,m\}\setminus\{i\}$, choose nonadjacent vertices
$a,b\in Y_j$. $y$ is adjacent to $a,b$, while $a,b$ are
anticomplete to $R_i$. Hence $(y,r_1r_2,\{a,b\})$ is an induced chair,
a contradiction. Therefore, $Y_j$ is a clique for every
$j\in\{0,\ldots,m\}\setminus\{i\}$.
\end{proof}

Since $x$ is complete to $Y$ and $Y_1,\ldots,Y_m$ are pairwise complete,
choosing one vertex from each $Y_i$ together with $x$ gives a clique of size
$m+1$. Hence $m\leq\omega-1$.

\begin{lemma}\label{lem1}
Let $i\in[m]$ and let $a,b\in Y_i$ be nonadjacent. Then the following holds.
\begin{enumerate}[(1)]
    \item $N_{R_i}(a)=N_{R_i}(b)$.
\item $N_{R_i}^3(a)=N_{R_i}^3(b)=\emptyset$. Moreover, for every
$j\in[m]\setminus\{i\}$, the set $Y_j$ is complete to $R_j$.
\end{enumerate}
\end{lemma}

\begin{proof}
(1) Suppose, without loss of generality, that some vertex $r\in V(R_i)$ is
adjacent to $a$ and nonadjacent to $b$. Then $(x,ar,\{b,v_0\})$ is an
induced chair, a contradiction.

It remains to show (2). By (1), $N_{R_i}(a)=N_{R_i}(b)$. Suppose that
$N_{R_i}^3(a)\neq\emptyset$. Then there is an induced path
$ar_1r_2r_3$, where $r_\ell\in N_{R_i}^{\ell}(a)$ for
$\ell\in\{1,2,3\}$. Since $N_{R_i}(a)=N_{R_i}(b)$, $b$ is
adjacent to $r_1$ and nonadjacent to $r_2$ and $r_3$. Hence
$(r_1,r_2r_3,\{a,b\})$ is an induced chair, a contradiction. Thus
$N_{R_i}^3(a)=\emptyset$, and by symmetry,
$N_{R_i}^3(b)=\emptyset$.

Now let $j\in[m]\setminus\{i\}$ and $y\in Y_j$. Suppose that $y$ is not
complete to $R_j$. Since $y$ has a neighbor in the connected graph $R_j$,
there are adjacent vertices $r_1,r_2\in V(R_j)$ such that
$yr_1\in E(G)$ and $yr_2\notin E(G)$. By
Lemma~\ref{lem:Y-pairwise-complete}, $y$ is adjacent to both
$a$ and $b$. Since $a,b\notin Y_j$, the definition of $Y_j$ implies that
$a$ and $b$ are anticomplete to $R_j$. Therefore,
$(y,r_1r_2,\{a,b\})$ is an induced chair, a contradiction. Hence $Y_j$
is complete to $R_j$.
\end{proof}

\begin{lemma}\label{lem:degeneracy}
For each $i\in[m]$, $R_i$ is $(2\omega-3)$-degenerate. Moreover,
if $Y$ is a clique, then $H_1$ is $(2\omega-3)$-degenerate.
\end{lemma}

\begin{proof}
Suppose that some induced subgraph $F$ of $R_i$ satisfies
$\delta(F)\geq2\omega-2$. Let $p_0p_1\cdots p_\ell$ be a shortest path
from $x$ to $F$ in $G[\{x\}\cup Y_i\cup V(R_i)]$, where $p_0=x$ and
$p_\ell\in V(F)$. Then $p_1\in Y_i$, $\ell\geq2$, and no vertex of $F$
is adjacent to $p_j$ for $j\leq\ell-2$.

If $|N_F(p_{\ell-1})\cap N_F(p_\ell)|\geq\omega-1$, then this intersection
contains nonadjacent vertices $a,b$; otherwise, together with
$p_{\ell-1}$ and $p_\ell$, it would form a clique of size at least
$\omega+1$. If $\ell=2$, then $(p_1,xv_0,\{a,b\})$ is an induced chair.
If $\ell\geq3$, then
$(p_{\ell-1},p_{\ell-2}p_{\ell-3},\{a,b\})$ is an induced chair.

We may therefore assume that
$|N_F(p_{\ell-1})\cap N_F(p_\ell)|\leq\omega-2$. Since
$d_F(p_\ell)\geq2\omega-2$,
$N_F(p_\ell)\setminus N_F(p_{\ell-1})$ has at least $\omega$ vertices and
contains nonadjacent vertices $a,b$; otherwise, together with $p_\ell$ it
would form a clique of size at least $\omega+1$. Then
$(p_\ell,p_{\ell-1}p_{\ell-2},\{a,b\})$ is an induced chair. Hence every
$R_i$ is $(2\omega-3)$-degenerate.

Now suppose that $Y$ is a clique. If $m=0$, then $H_1=G[Y]$, and the result
is immediate. Assume that $m\geq1$. For every $y\in Y_i$,
$N_{R_i}(y)$ is a clique; otherwise, nonadjacent vertices $r,s$ in this set
give the induced chair $(y,xv_0,\{r,s\})$. Thus
$|N_{R_i}(y)|\leq\omega-1$. Also, $|Y|\leq\omega-1$ because $x$ is
complete to $Y$.

Let $F'$ be a nonempty induced subgraph of $H_1$. If
$V(F')\cap Y\neq\emptyset$, choose $y\in V(F')\cap Y$. If $y\in Y_0$,
then $d_{F'}(y)\leq|Y|-1\leq\omega-2$. If $y\in Y_i$ for some $i\in[m]$,
then
$d_{F'}(y)
\leq
|Y|-1+|N_{R_i}(y)|
\leq
2\omega-3$.
If $V(F')\cap Y=\emptyset$, then $F'$ is contained in the pairwise
anticomplete union of $R_1,\ldots,R_m$, and the first part gives a vertex
of degree at most $2\omega-3$ in $F'$. Therefore, $H_1$ is
$(2\omega-3)$-degenerate.
\end{proof}

\section{Proof of Theorem~\ref{main thm}}

We restate the theorem for convenience.

\twolinearlydivisible*

\begin{proof}[Proof of Theorem~\ref{main thm}]
By definition, it suffices to prove that every connected chair-free graph
admits a $(2,2)$-linear division. We proceed by induction on $|V(G)|$.
Let $G$ be a connected chair-free graph, let $\omega=\omega(G)$, and
assume that every connected chair-free graph with fewer vertices admits a
$(2,2)$-linear division. Thus, every proper induced subgraph of $G$ is
$(2,2)$-linearly divisible.

If $\omega\leq2$, take $G_1$ to be empty, $G_2=G$, $r=0$, and
$\alpha=\omega$. The case $\omega=1$ is immediate, while every
triangle-free chair-free graph is
$3$-colorable~\cite{Randerath-Schiermeyer2002}. Thus,
$\chi(G)\leq\omega(\omega+1)$, and these choices give a $(2,2)$-linear
division. Hence, we may assume that $\omega\geq3$.

Apply the global setup from Section~\ref{section:structure}. Let $v_0$ be a
vertex contained in a maximum clique of $G$. If $G-N[v_0]=\emptyset$, take
$G_1=D_1=G[N(v_0)]$, $G_2=G[\{v_0\}]$, $r=1$, and $\alpha=1$. Then
$\omega(D_1)=\omega-1\geq2$ and $\chi(G_2)=1\leq2\omega$, so these
graphs give a $(2,2)$-linear division.

Assume that $G-N[v_0]\neq\emptyset$, and let $H_1,\ldots,H_t$ be its
components, where $H_1$ has maximum chromatic number. If $H_1$ is a clique,
take $G_1=D_1=G[N(v_0)]$ and
$G_2=G\left[\{v_0\}\cup\bigcup_{i=1}^tV(H_i)\right]$.
Since the graphs $H_i$ are pairwise anticomplete and $v_0$ is anticomplete
to each of them,
$\chi(G_2)=\max_{i\in[t]}\chi(H_i)=\chi(H_1)\leq\omega$. Thus, $r=1$
and $\alpha=1$ give a $(2,2)$-linear division. Hence $H_1$ is not a clique, and we
use the notation $t'$ from the global setup.

Choose $x,Y,R_1,\ldots,R_m,Y_0,Y_1,\ldots,Y_m$ as in
Section~\ref{section:structure}.

\medskip
\textbf{We first assume that $m=0$.}
Let
\[
U_1=\{u\in N(v_0):u\text{ is complete to }H_1\},
\qquad
U_2=\{u\in N(v_0):u\text{ is anticomplete to }H_1\}.
\]
By the choice of $x$, no vertex of $N(v_0)$ is mixed on $H_1$. Hence
$N(v_0)=U_1\sqcup U_2$, and $x\in U_1$.

$U_1$ is complete to $U_2$. Otherwise, choose nonadjacent vertices
$a,b\in V(H_1)$ and nonadjacent vertices $u\in U_1$, $z\in U_2$. Then
$(u,v_0z,\{a,b\})$ is an induced chair, a contradiction.

Every vertex of $U_1$ is either complete or anticomplete to each $H_i$.
Indeed, if $u\in U_1$ is mixed on $H_i$ for some $i\geq2$, connectedness of
$H_i$ gives adjacent vertices $p,q\in V(H_i)$ such that $up\in E(G)$ and
$uq\notin E(G)$. Choosing nonadjacent vertices $a,b\in V(H_1)$ gives the
induced chair $(u,pq,\{a,b\})$, a contradiction.

For $u\in U_1$, let
$I(u)=\{i\in[t']:u\text{ is complete to }H_i\}$. We show that the
sets $I(u)$, $u\in U_1$, form a chain. First, suppose that
$I(u)\neq I(w)$. After interchanging $u$ and $w$, choose $i\in I(u)\setminus I(w)$. Since $1\in I(u)\cap I(w)$, we have $i\neq1$. Choose nonadjacent vertices $a,b\in V(H_i)$ and a vertex $c\in V(H_1)$. If $uw\notin E(G)$, then $(u,cw,\{a,b\})$ is an induced chair, a contradiction. Thus, $uw\in E(G)$.

If $I(u)$ and $I(w)$ are incomparable, choose
$i\in I(u)\setminus I(w)$ and $j\in I(w)\setminus I(u)$. For nonadjacent
vertices $a,b\in V(H_i)$ and any $h\in V(H_j)$, the vertices
$(u,wh,\{a,b\})$ induce a chair, a contradiction. Hence, the sets
$I(u)$, $u\in U_1$, form a chain.

Moreover, $\bigcup_{u\in U_1}I(u)=[t']$. Let
$i\in[t']\setminus\{1\}$. Since $G$ is connected, some vertex
$z\in N(v_0)$ has a neighbor $h\in V(H_i)$. If no vertex of $U_1$ is
complete to $H_i$, then $U_1$ is anticomplete to $H_i$, so $z\in U_2$.
Choose $u\in U_1$ and nonadjacent vertices $a,b\in V(H_1)$. Since
$uz\in E(G)$, the vertices $(u,zh,\{a,b\})$ induce a chair, a
contradiction.

Since $\{I(u):u\in U_1\}$ is a finite chain, it has largest member $[t']$.
Let $U^*=\{u\in U_1:I(u)=[t']\}$, and let
\[
D_1=G[U^*], \qquad D_2=G\left[(U_1\setminus U^*)\cup U_2
       \cup\bigcup_{i=1}^{t'}V(H_i)\right], \qquad
G_2=G\left[\{v_0\}\cup\bigcup_{i=t'+1}^tV(H_i)\right].
\]
Then $V(G)=V(D_1)\sqcup V(D_2)\sqcup V(G_2)$. Both $D_1$ and $D_2$
are nonempty and complete to one another. Indeed,
$U^*$ is complete to $U_2$ and to every $H_i$ with $i\leq t'$, while a
vertex of $U^*$ is adjacent to every vertex of $U_1\setminus U^*$ because
their sets $I(u)$ are distinct. The graph $G_2$ is a disjoint union of
cliques and the isolated vertex $v_0$, so $\chi(G_2)\leq\omega\leq
2\omega-1$. Corollary~\ref{cor:two-blocks} shows that $G$ is
$(2,2)$-linearly divisible.

\medskip

\textbf{Now we may assume that $m\geq1$.} We may also assume that $Y$ is not a
clique. Otherwise, Lemma~\ref{lem:degeneracy} gives
$\chi(H_1)\leq2\omega-2$. Take $G_1=D_1=G[N(v_0)]$ and
$G_2=G\left[\{v_0\}\cup\bigcup_{i=1}^tV(H_i)\right]$.
Then $\chi(G_2)=\chi(H_1)\leq2\omega-2$. Thus, $r=1$ and $\alpha=1$
give a $(2,2)$-linear division.

\medskip
\noindent\textbf{Case when $t'\geq2$.}
By Lemma~\ref{t>1}, there is a vertex $w\in N(v_0)$ complete to
$\bigcup_{i=1}^{t'}V(H_i)$. Choose nonadjacent vertices $a,b\in Y$. The
vertices $w,a,b$ induce a $2$-template. By
Lemma~\ref{lem:template-extension}, this template extends to a maximal
template $Q$ containing $w$ and some vertex $z\in\{a,b\}$.

For $i\leq t'$, every vertex of $H_i$ is adjacent to $w$, and hence lies in
$V(Q)\cup N(Q)$. For each $j\in\{t'+1,\ldots,t\}$, the preliminary
argument in the proof of Lemma~\ref{t>1} gives a vertex
$w_j\in N(v_0)$ complete to $H_1\cup H_j$. Since
$z\in V(H_1)\cap V(Q)$, $w_j$ lies in $V(Q)\cup N(Q)$, and
every vertex of $H_j$ has distance at most two from $Q$. Also,
$v_0\in N(Q)$ because $w\in V(Q)$, so every vertex of $N(v_0)$ has
distance at most two from $Q$. Therefore, $N^{\geq3}(Q)=\emptyset$.

Let $U=N^2(Q)\cap N(v_0)$. Then $N^2(Q)
\subseteq
U\cup\bigcup_{j=t'+1}^tV(H_j)$.
The set $U$ is a clique. Indeed, if $u_1,u_2\in U$ are nonadjacent, then
$(v_0,wz,\{u_1,u_2\})$ is an induced chair. Moreover, $U$ is anticomplete to every $H_j$ with
$j>t'$. Indeed, if $u\in U$ has a neighbor in $H_j$, the preliminary
argument in the proof of Lemma~\ref{t>1} shows that
$u$ is complete to $H_1$, contradicting $uz\notin E(G)$. Finally,
$|U|\leq\omega-1$, and $|V(H_j)|\leq\omega-1$ because $w_j$ is complete
to the clique $H_j$. Thus, $G[N^2(Q)]$ is a disjoint union of cliques of
order at most $\omega-1$, and
$\chi(G[N^2(Q)])\leq\omega-1$. Lemma~\ref{lem:template-divisible} now
shows that $G$ is $(2,2)$-linearly divisible.

\medskip

Hence we may assume that $t'=1$. For later use, note that for every
$j\in\{2,\ldots,t\}$, the same preliminary argument gives a vertex
$w_j\in N(v_0)$ complete to $H_1\cup H_j$. In particular,
$|V(H_j)|\leq\omega-1$. Moreover, if a template $Q$ contains $x$ and a
vertex $z\in Y$, then every vertex of $N(v_0)$ and every vertex of
$H_2\cup\cdots\cup H_t$ has distance at most two from $Q$.
$N^2(Q)\cap N(v_0)$ is a clique and is anticomplete to
$H_2\cup\cdots\cup H_t$. Indeed, nonadjacent vertices
$u_1,u_2\in N^2(Q)\cap N(v_0)$ give the induced chair
$(v_0,xz,\{u_1,u_2\})$, while a vertex of $N^2(Q)\cap N(v_0)$ with a
neighbor in some $H_j$ is complete to $H_1$, contradicting its
nonadjacency to $z$.

We distinguish three cases.

\medskip
\noindent\textbf{Case 1.}
There are at least two indices $i\in[m]$ for which $R_i$ is not complete to
$Y_i$. Relabel so that this holds for $i=1,2$. Applying
Lemma~\ref{lem:Y-pairwise-complete} first with $i=1$ and then with $i=2$
shows that every set $Y_0,Y_1,\ldots,Y_m$ is a clique. Since these sets are
pairwise complete, $Y$ is a clique, a contradiction.

\medskip
\noindent\textbf{Case 2.}
For every $i\in[m]$, $R_i$ is complete to $Y_i$. Each $R_i$ is a
clique. Otherwise, for nonadjacent vertices $r_1,r_2\in V(R_i)$ and any $y_i\in Y_i$, the vertices $(y_i,xv_0,\{r_1,r_2\})$ induce a chair, a
contradiction.

Since $Y$ is not a clique and $Y_0,Y_1,\ldots,Y_m$ are pairwise complete,
some $Y_s$, where $s\in\{0,\ldots,m\}$, is not a clique. Choose
nonadjacent vertices $a,b\in Y_s$, and for each
$i\in[m]\setminus\{s\}$ choose a vertex $y_i\in Y_i$. The singleton parts
$\{x\}$ and $\{y_i\}$, together with the last part $\{a,b\}$, form a
template. By Lemma~\ref{lem:template-extension}, it extends to a maximal
template $Q$ containing $x$, every $y_i$, and some $z\in\{a,b\}$.

Every vertex of $Y$ is adjacent to $x$. Each $R_i$ has a neighbor in $Q$:
use $z$ when $i=s\in[m]$, and use $y_i$ otherwise. Hence
$V(H_1)\subseteq V(Q)\cup N(Q)$. The preceding observation shows that
every other vertex of $G$ has distance at most two from $Q$, so
$N^{\geq3}(Q)=\emptyset$.

We have $N^2(Q)
\subseteq
\bigl(N^2(Q)\cap N(v_0)\bigr)
\cup
\bigcup_{j=2}^tV(H_j)$.
$N^2(Q)\cap N(v_0)$ is a clique anticomplete to every $H_j$, and
each of these cliques has order at most $\omega-1$. Therefore,
$\chi(G[N^2(Q)])\leq\omega-1$, and Lemma~\ref{lem:template-divisible}
shows that $G$ is $(2,2)$-linearly divisible.

\medskip
\noindent\textbf{Case 3.}
There is exactly one index $i\in[m]$ for which $R_i$ is not complete to
$Y_i$. Relabel so that this index is $1$. By
Lemma~\ref{lem:Y-pairwise-complete}, the sets $Y_0,Y_2,\ldots,Y_m$ are
cliques. Since $Y$ is not a clique, $Y_1$ is not a clique. For every
$i\in\{2,\ldots,m\}$, the graph $R_i$ is also a clique; otherwise, for
nonadjacent vertices $r_1,r_2\in V(R_i)$ and any $y_i\in Y_i$, the
vertices $(y_i,xv_0,\{r_1,r_2\})$ induce a chair, a contradiction.

Choose nonadjacent vertices $a,b\in Y_1$. By Lemma~\ref{lem1}, the
sets $N_{R_1}(a)$ and $N_{R_1}(b)$ are equal and nonempty. Moreover,
$N_{R_1}^3(a)=N_{R_1}^3(b)=\emptyset$. For each
$i\in\{2,\ldots,m\}$, choose a vertex $y_i\in Y_i$. The singleton parts
$\{x\},\{y_2\},\ldots,\{y_m\}$, together with the last part $\{a,b\}$,
form a template. Extend it to a maximal template $Q$ by
Lemma~\ref{lem:template-extension}, and choose
$z\in\{a,b\}\cap V(Q)$. Thus, $x,z,y_2,\ldots,y_m\in V(Q)$.

Every vertex of $Y$ is adjacent to $x$, and every vertex of $R_i$ with
$i\geq2$ is adjacent to $y_i$. Moreover, $N_{R_1}^3(z)=\emptyset$, so
every vertex of $R_1$ has distance at most two from $Q$. The preceding
observation handles $N(v_0)$ and $H_2,\ldots,H_t$. Therefore,
$N^{\geq3}(Q)=\emptyset$.

Let $A=N_{R_1}(z)$ and $B=V(R_1)\setminus A$.

\begin{claim}\label{Bclique}
Every component of $G[B]$ is a clique.
\end{claim}

\begin{proof}
Let $B_i$ be a component of $G[B]$. Since $R_1$ is connected, some vertex
of $A$ has a neighbor in $B_i$. Every vertex $u\in A$ is either complete or
anticomplete to $B_i$. Otherwise, connectedness of $B_i$ gives adjacent
vertices $b_1,b_2\in B_i$ such that $ub_1\in E(G)$ and
$ub_2\notin E(G)$. By Lemma~\ref{lem1}(1), $u$ is complete to
$\{a,b\}$, and $(u,b_1b_2,\{a,b\})$ is an induced chair, a contradiction.
Choose $u\in A$ with a neighbor in $B_i$. Then $u$ is complete to $B_i$.
If $B_i$ contains nonadjacent vertices $b_1,b_2$, then
$(u,zx,\{b_1,b_2\})$ is an induced chair, a contradiction. Hence $B_i$ is a
clique.
\end{proof}

Let $U_0=N^2(Q)\cap N(v_0)$. Since $Y\cup A$ and
$R_2\cup\cdots\cup R_m$ lie in $V(Q)\cup N(Q)$,
$N^2(Q)\setminus U_0
\subseteq
B\cup\bigcup_{j=2}^tV(H_j)$.
Thus, every component of $G[N^2(Q)\setminus U_0]$ is a clique, by
Claim~\ref{Bclique} and the fact that $H_2,\ldots,H_t$ are pairwise
anticomplete cliques. Let $F_1,\ldots,F_s$ be these components. By the
preceding observation, $U_0$ is a clique, while $F_1,\ldots,F_s$ are
pairwise anticomplete cliques.

Fix a proper coloring of $U_0$ from a palette of $\omega$ colors. For each $i\in[s]$, since $U_0$ and $V(F_i)$ are cliques, the complement
of $G[U_0\cup V(F_i)]$ is bipartite. Hence
$G[U_0\cup V(F_i)]$ is perfect by the Weak Perfect Graph Theorem.
It therefore admits a coloring with at most $\omega$ colors. Since $U_0$ is
a clique, after permuting the colors we may assume that this coloring agrees
with the fixed coloring of $U_0$. As $F_1,\ldots,F_s$ are pairwise
anticomplete, these colorings combine to give
$\chi(G[N^2(Q)])\leq\omega$. Lemma~\ref{lem:template-divisible} now shows
that $G$ is $(2,2)$-linearly divisible.
\end{proof}

\section*{Acknowledgments}
We thank Songling Shan for an insightful question that led us to modify the
definition of $(k,\ell)$-linear divisibility to take into account the number of
parts in the decomposition. The Paley graph $P(17)$ was identified during
exploratory use of OpenAI's ChatGPT and was subsequently verified by the
authors. OpenAI's ChatGPT was also used for language polishing and for
generating the Python code used in the computational verification. All
mathematical arguments and computational results were checked by the authors.

\end{document}